\documentclass{amsart}
\usepackage{amsmath, amssymb}
\usepackage[numbers]{natbib}
\usepackage{hyperref}
\usepackage{enumitem}
\setlist{nosep}
\newtheorem{theorem}{Theorem}[section]
\theoremstyle{observation}
\newtheorem{observation}[theorem]{Observation}
\newtheorem{proposition}[theorem]{Proposition}
\newtheorem{lemma}[theorem]{Lemma}
\newtheorem{corollary}[theorem]{Corollary}
\theoremstyle{definition}
\newtheorem{definition}[theorem]{Definition}
\newtheorem{example}[theorem]{Example}
\newtheorem{conjecture}[theorem]{Conjecture}

\theoremstyle{remark}

\numberwithin{equation}{section}

\begin{document}
\title[TPP triples \& finite groups with abelian normal subgroups of prime index]{A note on the triple product property for finite groups with abelian normal subgroups of prime index}
\author{Sandeep R. Murthy}
\address{Dorchester-on-Thames, Oxfordshire, United Kingdom}
\curraddr{}
\email{srm@tuta.com}
\thanks{The author is unaffiliated. With the exception of some minor updates, all the work contained in this note was originally completed in 2013-2014, and for this the author is grateful for the assistance of Prof. Sarah Hart, Birkbeck, University of London, and Dr. Peter M. Neumann, The Queen's College, Oxford.}
\date{11 December 2025}
\dedicatory{}

\begin{abstract}
{Three non-empty subsets $S,T,U$ of a group $G$ are said to satisfy the triple product property (TPP) if, for elements $s,s' \in S$, and $t,t' \in T$, and $u,u' \in U$, the equation $s's^{-1}t't^{-1}u'u^{-1}=1$ holds if and only if $s = s'$, $t = t'$, $u = u'$.  If this is the case then $(S,T,U)$ is called a TPP triple of $G$ and $|S||T||U|$ the size of the triple.  If $G$ is a finite group the triple product ratio of $G$ can be defined as the quantity $\rho(G) := \frac{\beta(G)}{|G|}$, where $\beta(G)$ is the largest size of a TPP triple of $G$, and a special case of this, the subgroup triple product ratio, is the quantity $\rho_0(G) := \frac{\beta_0(G)}{|G|}$, where $\beta_0(G)$ is the largest size of a TPP triple of $G$ composed only of subgroups. There is a conjecture that $\rho(G) \leq \frac{4}{3}$ if $G$ contains a cyclic subgroup of index $2$ \citep[Conjecture 7.6]{HM}.  This note proves a more general version of this conjecture for subgroups by showing that $\rho_0(G) \leq \frac{p^2}{2p-1}$ if $G$ is any finite group that contains an abelian normal subgroup of prime index $p$, an improvement by a factor of $\frac{1}{2p-1}$ on the general upper bound of $p^2$ when $G$ contains any abelian subgroup of index $p$. In conclusion a generalised conjecture using the same upper bound is presented for $\rho$ for groups with cyclic normal subgroups of prime index, based on the known data for $\rho$ in such groups of small order.}
\end{abstract}

\maketitle

\tableofcontents

\section{Introduction}

Notation: Standard set-theoretic and group-theoretic notation is used. Groups will generally be finite, unless otherwise stated.

\begin{definition}\citep[Definition 2.1]{CU} Let $G$ be a group, finite or infinite, and $S,T,U$ non-empty subsets of $G$ with cardinalities $|S|, |T|, |U|$ respectively.  The triple $(S,T,U)$ is said to satisfy the \emph{triple product property} (TPP) if

\begin{equation}\label{GeneralTPP} s's^{-1}t't^{-1}u'u^{-1} = 1 \Longrightarrow s = s', t = t', u = u' \end{equation}

for all $s, s' \in S$, $t,t' \in T$, $u,u' \in U$. In this case, $G$ is said to \emph{realise} a TPP triple of \emph{parameter type}, or simply \emph{type}, $(|S|, |T|, |U|)$, and $|S|, |T|, |U|$ are called the \emph{parameters} of the triple and the product $|S||T||U|$ the \emph{size} of the triple.  If, additionally, $S, T, U$ are subgroups of $G$ then $(S,T,U)$ is called a \emph{subgroup TPP triple} of $G$, in which case the defining relation above simplifies to

\begin{equation} stu = 1 \Longrightarrow s = t = u = 1, \end{equation}

for all $s \in S$, $t \in T$, $u \in U$.
\end{definition}

TPP triples were first introduced by H. Cohn and C. Umans in 2003 to study the complexity of fast matrix multiplication in the context of finite groups, to be more specific, to realise matrix multiplication as multiplication in a finite group algebra via a triple of non-empty subsets of the group satisfying the TPP, and that are used to index the rows and columns of the matrices being multiplied and then recover the entries of the product \citep{CU}. The effectiveness of this group-theoretic approach depends on finding the smallest possible groups that realise TPP triples with parameter types matching the dimensions of matrix multiplication of interest \citep{CUKS, HM, HHMM}. 

\vspace{1mm}

However this note pursues some combinatorial aspects of maximising TPP triple sizes in finite groups, independently of their applications to matrix multiplication, as described more formally by P. Neumann in \citep{Neu}.  These involve two quantities. The first quantity, called the \emph{subgroup TPP ratio of $G$}, is defined as

\begin{equation}
\rho_0(G) := \frac{\beta_0(G)}{|G|}
\end{equation}

where $\beta_0(G)$ is the \emph{subgroup TPP capacity of $G$} defined as

\begin{equation}
\beta_0(G) := \max \;\{|S||T||U| \mid (S,T,U) \textrm{ is a subgroup TPP triple of }G\}.
\end{equation}.

The second quantity, called the \emph{TPP ratio of $G$}, is defined as

\begin{equation}
\rho(G) := \frac{\beta(G)}{|G|}
\end{equation}

where $\beta(G)$ is the \emph{TPP capacity of $G$} defined as

\begin{equation}
\beta(G) := \max \;\{|S||T||U| \mid (S,T,U) \textrm{ is a TPP triple of }G\}. 
\end{equation}

As subgroup TPP triples are special cases of TPP triples that are composed only of subgroups, clearly $\beta_0(G) \leq \beta(G)$ and $\rho_0(G) \leq \rho(G)$.  Note that $G$ always realises the trivial subgroup TPP triple $(G,\{1\},\{1\})$, so that $\beta_0(G) \geq |G|$, or equivalently, $\rho_0(G) \geq 1$.

\vspace{1mm}

It is natural to seek best upper bounds for $\beta(G)$ and $\beta_0(G)$ (equivalently, for $\rho(G)$ and $\rho_0(G)$). It was shown by Cohn and Umans that if $G$ is a dihedral group then $\rho(G) \geq \frac{4}{3}$ \citep{CU}. P. Neumann derived a general upper bound for $\beta(G)$ \citep[Corollary 3.2]{Neu} that

\begin{equation}
\beta(G) \leq \left( \frac{1 + \sqrt{1 + 8|G|}}{4}  \right)^3 < |G|^{3/2} 
\end{equation}

using the fact that \citep[Observation 3.1]{Neu}

\begin{observation}
If $(S, T, U)$ is a TPP triple of a group $G$ then $|S|(|T| + |U| - 1) \leq |G|$.
\end{observation}

In 2012 Hedtke and Murthy conjectured \citep[Table 1 and Conjectures 7.5-7.6]{HM}, based on tables of TPP triple data in groups with cyclic subgroups of index $2$ of order up to $32$, obtained via exhaustive search algorithms in these groups implemented using the GAP computer algebra system, that

\begin{conjecture}
If $G$ is a group with a cyclic subgroup of index $2$ then $\rho(G) \leq \frac{4}{3}$.
\end{conjecture}

The author is aware of a proof (from a private communication from another researcher) of this conjecture for dihedral groups $D_{2n}$, that is, $\rho(D_{2n}) \leq \frac{4}{3}$.

\vspace{1mm}

This note focuses on subgroup TPP ratio $\rho_0$ (equivalently, subgroup TPP capacity $\beta_0$), and proves a generalisation of the conjecture for $\rho_0$ for all groups with abelian normal subgroups of prime index. First, a few relevant properties and characterisations of TPP triples are stated, followed by proofs of two technical results needed to establish the main result. Following the proof of the main result and a basic corollary some implications for $\rho$ for groups with cyclic normal subgroups of prime index are discussed, based on known data for $\rho$ in such groups.

\section{Elementary Properties of TPP Triples}

Let $G$ be a group.

\begin{definition}\citep[p. 234]{Neu}\label{BasicTPPTriple} A TPP triple $(S,T,U)$ of $G$ is called \emph{basic} if
\begin{equation}\label{BasicTPPC1}S \cap T \cap U = \{1\}.
\end{equation}
\end{definition}

The TPP can also be described in terms of certain sets called \emph{quotient sets}, which have the following definition.

\begin{equation}
Q(X, Y) = XY^{-1} = \{xy^{-1} | x \in X, y \in Y\} \subseteq G, \hspace{3em}X, Y \subseteq G; X, Y \neq \emptyset
\end{equation}

If $X = Y$ then $Q(X)$ is a shorthand for $Q(X, Y)$.

\vspace{1mm}

Note that, by definition, a quotient set $Q(X)$ contains the identity $1$, is equal to its inverse, that is, $Q(X) = Q(X)^{-1}$, and, furthermore, if $1 \in X$ then $Q(X)$ contains both $X$ and its inverse $X^{-1}$, that is, $X \cup X^{-1} \subseteq Q(X)$. However, a quotient set is not necessarily closed under taking products, as otherwise it would be a (sub)group, which isn't generally true. For a subgroup it is true that it is equal to its quotient set.

\begin{theorem} \label{HM} \textup{\citep[Theorem 3.1]{HM}} Three non-empty subsets $S,T,U \subseteq G$ satisfy the TPP if and only if 
\begin{equation}\label{BasicTPPC2}Q(S) \cap Q(T)Q(U)  = Q(T) \cap Q(U) = \{1\}.
\end{equation}

If $S, T, U$ happen to be subgroups then the defining relation above simplifies to

\begin{equation}\label{SubgroupTPP}S \cap TU  = T \cap U = \{1\}.
\end{equation}
\end{theorem}

The following result states two invariance properties of TPP triples under permutations or certain kinds of set translations of its members.

\begin{observation} \label{Invariance} Let $(S,T,U)$ be any TPP triple of $G$, not necessarily basic in the sense defined above.
\begin{enumerate}
\vspace{1mm}
\item If $\pi \in S_3$, that is, $\pi$ is a permutation of $\{S, T, U\}$, then $(S^\pi,T^\pi,U^\pi)$ is a TPP triple of $G$ (permutation invariance) \textup{\citep[Lemma 2.1]{CU}}.
\vspace{1mm}
\item If $a,b,c,d \in G$ are any elements then $(dSa,dTb,dUc)$ is a TPP triple of $G$  (translation invariance) \textup{\citep[Observation 2.1]{Neu}}.
\end{enumerate}
\end{observation}

The translation invariance property for TPP triples means that any non-basic TPP triple can be translated to a basic TPP triple of the same type and size as the original triple.  To be precise, if $(S,T,U)$ is any non-basic TPP triple of $G$, that is, when $1 \notin S \cap T \cap U$, then elements $s \in S$, $t \in T$, $u \in U$ can be chosen such that the right-translated triple $(Ss^{-1},Tt^{-1},Uu^{-1})$ is a basic TPP triple of $G$ with parameters $|Ss^{-1}|=|S|$, $|Tt^{-1}|=|T|$, $|Uu^{-1}|=|U|$.  This means that generally only basic TPP triples need be considered.  Of course, subgroup TPP triples are necessarily basic.

\vspace{1mm}

All further references to TPP triples will be to basic TPP triples, unless otherwise stated.

\vspace{1mm}

The following two elementary properties are useful to note, and will be used at several points.

\begin{observation}\label{SetClosureToSubgroups} Given a TPP triple $(S, T, U)$ of a group $G$ and a subgroup $H \leq G$ every triple of the form $(S' \cap H, T' \cap H, U' \cap H)$ for non-empty subsets $S' \subseteq S$, $T' \subseteq T$, $U' \subseteq U$, with $1 \in S' \cap T' \cap U'$, is a TPP triple of $H$.
\end{observation}
\begin{proof}If $G$ and $H$ are as given, and $(S, T, U)$ is a TPP triple of $G$, then for any non-empty subsets $S' \subseteq S$, $T' \subseteq T$, $U' \subseteq U$, with $1 \in S' \cap T' \cap U'$, define the subsets $S'_0 := S' \cap H$, $T'_0 := T' \cap H$, $U'_0 := U' \cap H$. Then $1 \in S'_0 \cap T'_0 \cap U'_0$ and $Q(S'_0) \subseteq Q(S)$, $Q(T'_0) \subseteq Q(T)$, $Q(U'_0) \subseteq Q(U)$, and $Q(S'_0) \cap Q(T'_0)Q(U'_0) = Q(T'_0) \cap Q(U'_0) = \{1\}$.
\end{proof}

\begin{proposition}\label{PairSizes}\textup{\citep[Lemma 3.1]{CU}} Let $(S,T,U)$ be a TPP triple of $G$.
\begin{enumerate}
\item If  $X,Y \in \{S,T,U\}$ and $X \neq Y$ then the mapping $(x,y) \longmapsto x^{-1}y$ on $X \times Y$ into $G$ is injective, and $|XY| = |X||Y| \leq |G|$, where the equality holds only if $Z = \{1\}$, where $Z \in \{S,T,U\}\backslash\{X,Y\}$.
\vspace{1mm}
\item If $G$ is abelian then the mapping $(s,t,u) \longmapsto s^{-1}tu$ on $S \times T \times U$ is injective into $G$, and $|S||T||U| \leq |G|$.
\end{enumerate}
\end{proposition}

Part \textup{(2)} of Proposition \ref{PairSizes} means that only nonabelian groups can realise TPP triples of non-trivial size.

\section{Coset Decomposition of TPP Triples}

Some technical results needed for the main result are stated and proved.  These are based on the idea of decomposing a TPP triple of a given group into smaller TPP triples obtained by independently decomposing the members of the triple with the left (or right) cosets of a suitable subgroup of \emph{small index}, as first described by P. Neumann in \citep[Observation 4.1]{Neu}.

\begin{definition}\label{CosetDecomp}Let $G$ be a group, $H$ a subgroup, $G/H = \{gH \mid g \in G\}$ the collection of all (left) cosets of $H$ in $G$, and $|G:H|$ the index (or size) of $G/H$. For a non-empty subset $S \subseteq G$ let the \emph{$H$-support} of $S$ be the set $\overline{S} := \{gH \in G/H \mid S \cap gH \neq \emptyset \} \subseteq G/H$, that is, the set of all cosets of $H$ that intersect with $S$. By definition $gH \in \overline{S} \iff S_g := S \cap gH \neq \emptyset$ for any coset $gH \in G/H$. If $(S,T,U)$ is a TPP triple of $G$ let its \emph{restriction to $H$}, or, simply, \emph{$H$-restriction}, be the triple $(S_0,T_0,U_0)$, where $S_0 := S \cap H$, $T_0 := T \cap H$, $U_0 := U \cap H$, and this is a TPP triple of $H$.
\end{definition}

\begin{observation}\label{CosetM}Let $G$ be a group and $H$ a subgroup.

\vspace{1mm}

\textup{(1)} If $S$ is a subgroup of $G$ and $\overline{S}$ is its $H$-support then $|\overline{S}| = |S : S \cap H|$.

\vspace{1mm}

\textup{(2)} If $S$ is a subgroup of $G$ and $H$ is normal in $G$ then $\overline{S}$ is a subgroup of $G/H$.

\vspace{1mm}

\textup{(3)}  If $H$ is abelian and normal in $G$, and $(S,T,U)$ is a subgroup TPP triple of $G$ then
\begin{equation}
|S||T||U| \leq \frac{\sigma \tau \upsilon}{n}|G|
\end{equation}

where $n = |G:H|$, and $\sigma = |S:S \cap H|$, $\tau = |T:T \cap H|$, $\upsilon = |U:U \cap H|$.
\end{observation}
\begin{proof}Let $G$ and $H$ be given as above and $n = |G : H|$.

\vspace{1mm}

\textup{(1)} Let $S \leq G$ and $\overline{S}$ be its $H$-support as defined above. Let $gH \in \overline{S}$ and $S_g := S \cap gH$. By definition $S_g \neq \emptyset$. If $x \in S_g$ then $|S_g| = |S \cap gH| = |xS \cap xH| = |x(S \cap H)| = |S \cap H|$, where $x(S \cap H)$ is a coset of $S \cap H \leq S$. As the cosets $gH \in G/H$ are disjoint so are the sets $S_g$, which thus form an equal-sized partition of $S$. If $\sigma = |\overline{S}|$ then $|S| = \sigma|S \cap H|$, i.e. $\sigma = |S: S \cap H|$. The same result holds if left cosets are replaced by right cosets.

\vspace{1mm}

\textup{(2)} Let $S \leq G$. If $H \trianglelefteq G$ then $G/H$ is a (quotient) group and using the Second Isomorphism Theorem $SH \leq G$ and $\overline{S} = SH/H \leq G/H$, where $SH/H \cong S/(S\cap H)$.

\vspace{1mm}

\textup{(3)} Let $H \trianglelefteq G$ be abelian, and $(S,T,U)$ be a subgroup TPP triple of $G$.  Then by Observation \ref{SetClosureToSubgroups} $(S \cap H, T \cap H, U \cap H)$ is a subgroup TPP triple of $H$.  If $S_0 := S \cap H$, $T_0 := T \cap H$, $U_0 := U \cap H$, and $\overline{S}, \overline{T}, \overline{U}$ are the $H$-supports of $S, T, U$ respectively, then by part (1) above $\sigma = |\overline{S}| = |S: S_0|$, $\tau = |\overline{T}| = |T: T_0|$, $\upsilon = |\overline{U}| = |U: U_0|$, where $\overline{S}, \overline{T}, \overline{U} \leq G/H$. Also, by assumption $H$ is abelian, so by Proposition \ref{PairSizes} it follows that $|S_0T_0U_0|=|S_0||T_0||U_0| = \frac{|S||T||U|}{\sigma \tau \upsilon} \leq |H|$, which shows that $|S||T||U| \leq \sigma 
\tau \upsilon |H| = \frac{\sigma\tau\upsilon}{n}|G|$. Note here that $\sigma, \tau, \upsilon$ are divisors of $n$, and of $|S|, |T|, |U|$ respectively, so $\sigma \leq \text{min}\{n, |S|\}$, $\tau \leq \text{min}\{n, |T|\}$, $\upsilon \leq \text{min}\{n, |U|\}$.
\end{proof}

Note that $\textup{(3.1)}$ can also be derived for the more general case of subset TPP triples, leading to the general upper bound
\begin{align}
|S||T||U| \leq n^2|G|
\end{align}
for the size of any TPP triple $(S, T, U)$ of a group $G$ with an abelian subgroup of index $n$ \citep[Corollary 4.2]{Neu}.

\begin{lemma}\label{DisjTPPSets}Let $G$ be a group, $H$ an abelian normal subgroup, $G/H$ the quotient group, and $(S,T,U)$ a subgroup TPP triple of $G$. Furthermore, let the $H$-supports $\overline{S}$, $\overline{T}$, $\overline{U}$ be defined as in Observation \ref{CosetM}, where these are subgroups of $G/H$, the subgroups $S_0 := S \cap H$, $T_0 := T \cap H$, $U_0 := U \cap H$, and the sets $S_x := S \cap xH$, $T_y := T \cap yH$, $U_z := U \cap zH$ for cosets $xH, yH, zH \in G/H$.

\vspace{1mm}

\textup{(1)} If $xH \in \overline{S} \cap \overline{T}$, then the set $S_x^{-1}T_xU_0$ is a coset of $S_0T_0U_0$ in $H$, and, moreover, non-trivial if $xH$ is non-trivial.

\vspace{1mm}

\textup{(2)} If $xH, yH \in \overline{S} \cap \overline{T}$ are distinct, \emph{and} $U_0$ is normal in $G$, then the cosets $S_x^{-1}T_xU_0$ and $S_y^{-1}T_yU_0$ are distinct.

\vspace{1mm}

\textup{(3)} If $xH \in \overline{S} \cap \overline{T}$ and $yH \in \overline{S} \cap \overline{U}$, and at least one of them is non-trivial, then the cosets $S_x^{-1}T_xU_0$ and $S_y^{-1}U_yT_0$ are distinct.
\vspace{1mm}

\end{lemma}

\begin{proof} Let $H$ be as given ($H$ is abelian and normal in $G$). Several basic facts used in the proof are recalled for convenience:

\begin{itemize}
\item The inverse of any coset $xH = Hx \in G/H$ is $(xH)^{-1} = x^{-1}H = Hx^{-1}$, and the inverse of any non-empty subset $B_x \subseteq xH$ is $B_x^{-1} \subseteq x^{-1}H = Hx^{-1}$ (by the normality of $H$ in $G$).
\item If $x_1, x_2 \in xH$ are any coset elements then $x_1x_2^{-1} \in H$ (if $x_1 = xh_1$ and $x_2 = xh_2$ for some $h_1, h_2 \in H$ then $x_1x_2^{-1} = xh_1h_2^{-1}x^{-1} = xx^{-1}h_3 = h_3 \in H$ for some $h_3 \in H$ (by the normality of $H$ in $G$).
\item Every permutation $(S^\pi, T^\pi, U^\pi)$ of the TPP triple $(S, T, U)$ of $G$, for $\pi \in S_3$, is also a TPP triple of $G$ (permutation invariance, Observation \ref{Invariance}).
\item The subgroups $S_0, T_0, U_0$ as defined above, which satisfy the TPP, are normal in $S, T, U$ respectively, and $SH/H \cong S/S_0$, $TH/H \cong T/T_0$, $UH/H \cong U/U_0$ (Second Isomorphism Theorem). This means that translates of the form $S_0s_x = s_xS_0, T_0t_y = t_yT_0, U_0u_z = u_zU_0$, for elements $s_x \in S \cap xH$, $t_y \in T \cap yH$, $u_z \in U \cap zH$ and cosets $xH, yH, zH \in G/H$, are cosets of $S_0, T_0, U_0$ respectively in $S, T, U$ respectively that intersect with the cosets $xH, yH, zH$ respectively, and contain the elements $s_x, t_y, u_z$ respectively, that is, $s_xS_0 = S_0s_x = S_x$, $t_yT_0 = T_0t_y = T_y$, $u_zU_0 = U_0u_z = U_z$.
\item As $H$ is abelian, the subgroups $S_0$, $T_0$, and $U_0$, are also normal in $H$ and thus form a set product $S_0T_0U_0$ that is a subgroup of $H$ of order $|S_0T_0U_0| = |S_0||T_0||U_0| \leq |H|$ (using the injective triple product map on $S_0 \times T_0 \times U_0$ into $H$ from Proposition \ref{PairSizes}).
\end{itemize}

\vspace{1mm}

\textup{(1)} Let $xH \in \overline{S} \cap \overline{T}$. Then, by definition, the sets $S_x, T_x$ are non-empty, and elements $s_x \in S_x$, $t_x \in T_x$ can be chosen.  As $s_x^{-1}t_x \in H$ a translate of $S_0T_0U_0$ in $H$ can be formed, which is $s_x^{-1}t_x(S_0T_0U_0) = S_0s_x^{-1}t_xT_0U_0 = S_x^{-1}T_xU_0$ and thus a coset of $S_0T_0U_0$ in $H$. Its size is $|S_x^{-1}T_xU_0| = |s_x^{-1}t_x(S_0T_0U_0)| = |S_0T_0U_0| = |S_0||T_0||U_0| \leq |H|$. If $xH \neq H$ (if $x \in G\backslash H$) then the coset $S_x^{-1}T_xU_0$ is non-trivial, that is, $S_x^{-1}T_xU_0 \cap S_0T_0U_0 = \emptyset$ if $xH \neq H$. To see this, suppose there are elements $s_x \in S_x$, $t_x \in T_x$, $u_0 \in U_0$ and $s_0 \in S_0$, $t_0 \in T_0$, $\widetilde{u}_0 \in U_0$, such that $s_x^{-1}t_xu_0 = s_0^{-1}t_0\widetilde{u}_0$. This can be rearranged as $s_xs_0^{-1}t_0\widetilde{u}_0u_0^{-1}t_x^{-1} = 1$. As $t_0$ and $\widetilde{u}_0u_0^{-1}$ commute in $H$ this can be further rearranged as $s_xs_0^{-1}\widetilde{u}_0u_0^{-1}t_0t_x^{-1} = 1$. However the TPP for $(S, U, T)$ implies that $s_0 = s_x$ and $t_0 = t_x$, which is a contradiction if $xH \neq H$.

\vspace{1mm}

Note here that by transposing $T$ and $U$ it can be shown that there are cosets of the form $s_x^{-1}u_x(S_0T_0U_0) = s_x^{-1}u_x(S_0U_0T_0) = S_x^{-1}U_xT_0$ in $H$, and which are non-trivial if $xH$ is non-trivial.

\vspace{1mm}

\textup{(2)} Let $xH, yH \in \overline{S} \cap \overline{T}$ be distinct, and let $U_0 \trianglelefteq G$. As in part \textup{(1)} elements $s_x \in S_x$, $t_x \in T_x$, $s_y \in S_y$ and $t_y \in T_y$ can be chosen to form cosets $s_x^{-1}t_x(S_0T_0U_0) = S_x^{-1}T_xU_0$ and $s_y^{-1}t_y(S_0T_0U_0) = S_y^{-1}T_yU_0$ of $S_0T_0U_0$ in $H$, but these are distinct by the distinctness of $xH, yH$. To see this, suppose there are elements $s_x \in S_x$, $t_x \in T_x$, $s_y \in S_y$, $t_y \in T_y$, $u_0, \widetilde{u}_0 \in U_0$ such that $s_x^{-1}t_xu_0 = s_y^{-1}t_y\widetilde{u}_0$. This can be rearranged as $s_xs_y^{-1}t_y\widetilde{u}_0u_0^{-1}t_x^{-1} = 1$. As $U_0$ is normal in $G$ there is a $\hat{u_0} \in U_0$ such that $\widetilde{u}_0u_0^{-1}t_x^{-1} = t_x^{-1}\hat{u}_0$. Substituting this into the equation above implies that $s_xs_y^{-1}t_yt_x^{-1}\hat{u}_0 = 1$. However the TPP for $(S, T, U)$ implies that $s_x = s_y$, $t_x = t_y$, a contradiction.

\vspace{1mm}

Note here that it is necessary to assume the cosets $xH, yH$ are distinct, and that $U_0$ is normal in $G$, unlike in other parts of the proof. Also note that by transposing $T$ and $U$ in the argument above, it can be shown that cosets of the form $S_x^{-1}U_xT_0$ and $S_y^{-1}U_yT_0$ are distinct if $xH, yH \in \overline{S} \cap \overline{U}$ are distinct, and $T_0$ is normal in $G$.

\vspace{1mm}

\textup{(3)} Suppose $xH \in \overline{S} \cap \overline{T}$ and $yH \in \overline{S} \cap \overline{U}$, and that $xH$ is non-trivial, that is, $xH \neq H$. As above there are cosets $S_x^{-1}T_xU_0$ and $S_y^{-1}U_yT_0$ of $S_0T_0U_0$ in $H$, but these are distinct. To see this, suppose there are elements $s_x \in S_x$, $ s_y \in S_y$, $t_x \in T_x$, $u_y \in U_y$, $t_0 \in T_0$, and $u_0 \in U_0$ such that $s_x^{-1}t_xu_0 = s_y^{-1}u_yt_0$, or $s_xs_y^{-1}u_yt_0u_0^{-1}t_x^{-1} = 1$.  Since $t_0$ and $u_0^{-1}$ commute in $H$ this can be rewritten as $s_xs_y^{-1}u_yu_0^{-1}t_0t_x^{-1} = 1$.  However the TPP for $S, U, T$ implies that $t_x = t_0$, a contradiction if $xH \neq H$. In the same way, the assumption that $yH$ is non-trivial would also lead to a similar contradiction. Note here that there is no need to assume distinctness of the cosets $xH, yH$, only that at least one of them is non-trivial.

\end{proof}

\section{The Main Result for Subgroup TPP Ratio}

\begin{theorem}\label{MABQuoPrime}If $G$ is a group with an abelian normal subgroup $H$ of prime index $p$ then
\begin{equation}
\rho_0(G) \leq \frac{p^2}{2p-1}
\end{equation}

where equality $\rho_0(G) = \frac{p^2}{2p-1}$ implies that $2p -1$ divides $|H|$, and $|G|=p(2p-1)m$ where $m$ is the order of the proper subgroup of $H$ associated with the $H$-restriction of the subgroup TPP triple of $G$ through which the equality is achieved.
\end{theorem}
\begin{proof} Let $G$ and the prime $p$ be as given. Let $H$ be an abelian normal subgroup of index $p$ (thus making it maximal in $G$).  Then $G/H$ is a cyclic group of order $p=|G:H| = \frac{|G|}{|H|}$ whose elements are $p$ (left) cosets $xH$, for $x \in G$. Let  $(S,T,U)$ be a subgroup TPP triple of $G$ with $H$-supports $\overline{S},\overline{T},\overline{U}$ as defined in Observation \ref{CosetDecomp}. By part (2) of that Observation $\overline{S}, \overline{T}, \overline{U} \leq G/H$, and letting $\sigma = |\overline{S}|$, $\tau = |\overline{T}|$, $\upsilon = |\overline{U}|$ then $\sigma, \tau, \upsilon \in \{1, p\}$. The sets $S_x, T_y, U_z$, as defined in Lemma \ref{DisjTPPSets} will also be used at certain points: these are defined as $S_x := S \cap xH$, $T_y := T \cap yH$, $U_z := U \cap zH$ for any $xH, yH, zH \in G/H$. By definition, if $xH \in G/H$ then $xH \in \overline{S} \iff S_x \neq \emptyset$, and there are analogues for $T_y$ and $U_z$.

\vspace{1mm}

Let $S_0 := S \cap H$, $T_0 := T \cap H$, $U_0 := U \cap H$. Then $(S_0, T_0, U_0)$ is a subgroup TPP triple of $H$ (by Observation \ref{SetClosureToSubgroups}) of size $|S_0||T_0||U_0| = |S_0T_0U_0| = \frac{|S||T||U|}{\sigma\tau\upsilon} \leq |H|$ and $|S||T||U| \leq \frac{\sigma\tau\upsilon}{p}|G|$ (by Proposition \ref{PairSizes} and Observation \ref{CosetM}). Note that $S_0T_0U_0 \trianglelefteq H$ (as $H$ is abelian), and by the Second Isomorphism Theorem, $S_0 \trianglelefteq S$, $T_0 \trianglelefteq T$, and $U_0 \trianglelefteq U$.

\vspace{1mm}

Observation \ref{CosetM} implies the general bound $|S||T||U| \leq \frac{\sigma \tau \upsilon}{p}|G|$ applies, independently of any special assumptions about $|S||T||U|$. If at least two of $\sigma, \tau, \upsilon$ are equal to $1$ then $|S||T||U| \leq \frac{\sigma \tau \upsilon}{p}|G| \leq |G|$, and there is nothing to prove. So to achieve $|S||T||U| > |G|$ there are only two cases to consider: \textup{(i)} $\sigma\tau\upsilon = p^2$ or \textup{(ii)} $\sigma\tau\upsilon = p^3$. 

\vspace{1mm}

\textup{(*)} Now the special assumption is introduced (for which a contradiction is sought): suppose $|S||T||U| > \frac{p^2}{2p-1}|G| = \frac{p^3}{2p-1}|H|$.

\vspace{1mm}

Consider case \textup{(i)}: let $\sigma = \tau = p$ and $\upsilon = 1$, without loss of generality (permutation invariance of $S, T, U$ by Observation \ref{Invariance}), which means that $S$ and $T$ are supported on all $p$ cosets of $H$, and $U$ only on $H$ ($U = U_0$). Then $|S_0||T_0||U_0| = \frac{|S||T||U|}{p^2} > \frac{p}{2p-1}|H|$ and $\frac{|H|}{|S_0||T_0||U_0|} = |H:S_0T_0U_0| < \frac{2p-1}{p} = 2 - \frac{1}{p} < 2$, that is,  $|H:S_0T_0U_0| = 1$ and $S_0T_0U_0 = H$. However, from the $p - 1$ non-trivial cosets of $H$ on which $S$ and $T$ are commonly supported any coset $xH$ and elements $s_x \in S_x$, $t_x \in T_x$ can be chosen to form a coset $s_x^{-1}t_x(S_0T_0U_0) = S_x^{-1}T_xU_0$ of $S_0T_0U_0$ in $H$. And, by Lemma \ref{DisjTPPSets} part \textup{(1)} this is non-trivial, that is, $S_x^{-1}T_xU_0 \neq S_0T_0U_0$, a contradiction. This argument also shows that when only two of the subgroups $S, T, U$ are supported on all $p$ cosets of $H$ then $|S||T||U| < \frac{p^2}{2p-1}|G|$ and $\rho_0(G)$ can never achieve $\frac{p^2}{2p-1}$.

\vspace{1mm}

Consider case \textup{(ii)}, that is, $\sigma = \tau = \upsilon = p$, meaning that $S, T, U$ are supported on all $p$ cosets of $H$, or, equivalently, $\overline{S} = \overline{T} = \overline{U} = G/H$ and the sets $S_x, T_y, U_z$, as defined above, are all non-empty. Importantly, here, the normality of $T_0$ and $U_0$ in $G$ is necessary, and, indeed, follows from the conditions here: to see this, note that as $T_0$ is normal in both $T$ and $H$ it is normal in the subgroup $\langle T,  H \rangle$ generated by them, and as $T$ is not contained in $H$, where $H$ is maximal in $G$, then $\langle T,  H \rangle = G$. In the same way, $U_0$ is normal in $G$.

\vspace{1mm}

Now, $|S_0||T_0||U_0| = \frac{|S||T||U|}{p^3} > \frac{1}{2p-1}|H|$, that is, $|H:S_0T_0U_0| < 2p - 1$ and $S_0T_0U_0$ has less than $2p - 1$ cosets in $H$. However, using Lemma \ref{DisjTPPSets} parts \textup{(2)-(3)} two collections of $p - 1$ non-trivial cosets each of $S_0T_0U_0$ in $H$ can be formed, namely, $\mathcal{A} = \{S_x^{-1}T_xU_0\}_{xH \in G/H \backslash \{H\}}$ and $\mathcal{B} = \{S_y^{-1}U_yT_0\}_{yH \in G/H \backslash \{H\}}$, that are mutually distinct, that is, $S_x^{-1}T_xU_0 \neq S_y^{-1}U_yT_0$ for any pair $S_x^{-1}T_xU_0 \in \mathcal{A}$, $S_y^{-1}U_yT_0 \in \mathcal{B}$. And this implies that the total number of cosets of $S_0T_0U_0$ in $H$, including the trivial coset $S_0T_0U_0$, is equal to $1 + 2(p - 1) = 2p - 1$, a contradiction.

\vspace{1mm}

Thus \textup{(*)} cannot hold and so $|S||T||U| \leq \frac{p^2}{2p-1}|G|$, as claimed.

\vspace{1mm}

For the final part of the theorem, suppose $|S||T||U| = \frac{p^2}{2p-1}|G|$. Then $|S_0||T_0||U_0| = \frac{1}{2p-1}|H|$ and $|H:S_0T_0U_0| = 2p-1$ (so $2p-1 \mathrel | |H|$ also). As $S_0T_0U_0 \leq G$ then $|G:S_0T_0U_0| = |G:H||H:S_0T_0U_0| = p(2p-1)$ and $|G| = p(2p-1)|S_0T_0U_0| = p(2p-1)|S_0||T_0||U_0|$.
\end{proof}

Note that the upper bound $\frac{p^2}{2p-1}$ here for $\rho_0$ is an improvement by a factor of $\frac{1}{2p-1}$ on the general upper bound of $p^2$ noted previously which holds when $G$ contains any abelian subgroup of index $p$ (note that $\frac{p^2}{2p-1} \longrightarrow \left(\frac{1}{2}p\right)^+$ as $p \longrightarrow \infty$).

\begin{example}\label{DihedralGroupExm}
A simple illustration of the theorem (and the idea of coset decomposition) is given below using the dihedral group $D_{2n}$ of order $2n$ where $n$ is divisible by $3$, that is, $n=3m$ for some positive integer $m$ (so that $|D_{2n}| = 6m$), and the equality $\rho_0(G) = \frac{p^2}{2p-1}$ is achieved for $p = 2$.

\vspace{1mm}

Using the standard presentation $\langle r, f \mathrel | r^n = r^{3m} = 1, f^2 = 1, frf = r^{-1} \rangle$ of $D_{2n}$, and letting $H$ be the cyclic subgroup $\langle r \rangle = \{r^0 = 1, r,\ldots,r^{3m-1}\}$ of order $n = 3m$ (and index $2$), and $fH = \{f, fr,\ldots, fr^{3m-1}\}$ its non-trivial coset, consider the subgroups 

\begin{align*}
S &= \langle r^3, f \rangle = \{r^{3i}f^j \mathrel | i \in \mathbb{Z}_m, j \in \mathbb{Z}_2\} \\
T &= \langle fr \rangle = \{1, fr\} \\
U &= \langle fr^2 \rangle = \{1, fr^2\}
\end{align*}

These subgroups are generated by the set $\{r^3, f\}$ and the involutions $fr$ and $fr^2$ respectively. The largest of these, $S$, is generated by two elements, one, $r^3$, of order $\frac{1}{3}n = m$, and the other, $f$, an involution, and so forms of a subgroup of order $\frac{2}{3}n = 2m$. Clearly, $T \cap U = \{1\}$, and it is easy to check that $TU = \{1, fr, fr^2, frfr^2 = r^{-1}r^2\} = \{1, fr, fr^2, r\}$ and $S \cap TU = \{1\}$. Thus $(S, T, U)$ is a (subgroup) TPP triple of $D_{2n}$ of type $(\frac{2}{3}n, 2, 2) = (2m, 2, 2)$ and size $\frac{8}{3}n = 8m$, thus achieving $\rho_0(D_{2n}) = \frac{4}{3}$.
\end{example}

\vspace{1mm}

There are subgroups $S_0 = S \cap H = \langle r^3 \rangle$, $T_0 = T \cap H = \{1\}$, $U_0 = U \cap H = \{1\}$, which means that the subgroup $S_0T_0U_0$ is just $S_0$ and is cylic of order $m$. It has $3$ cosets in $H$, namely, $S_0$, $rS_0 = r\langle r^3 \rangle$, $r^2S_0 = r^2\langle r^3 \rangle$, and $3$ cosets in $G \backslash H$ (that is, $fH$), which are $fS_0 = f\langle r^3 \rangle = S_1$, $frS_0 = fr\langle r^3 \rangle$, $fr^2S_0 = fr^2\langle r^3 \rangle$. The non-trivial cosets of $T_0$ and $U_0$ are $T_1 = \{fr\}$ and $U_1 = \{fr^2\}$ respectively. There is a coset decomposition of the size of $(S, T, U)$ in terms of $S_0T_0U_0 = S_0$ given by

\begin{align*}|S||T||U| &= (|S_0| + |S_1|)(|T_0| + |T_1|)(|U_0| + |U_1|) \\
                        &= \sum_{0 \leq i, j, k \leq 1}|S_i||T_j||U_k| \\
                        &= \sum_{0 \leq i, j, k \leq 1} |s_i^{-1}t_ju_k(S_0T_0U_0)| \\
                        &= \sum_{0 \leq i, j, k \leq 1} |S_0||T_0||U_0| \\
                        &= \sum_{0 \leq i, j, k \leq 1} |S_0| \\
                        &= 8m = \frac{8}{3}n = \frac{4}{3}|D_{2n}|
\end{align*}

where $|S_i||T_j||U_k| = |S_i| = |S_0|$ is the size of the TPP triple $(S_i, T_j, U_k)$, and elements $s_i \in S_i$, $t_j \ \in T_j$, $u_k \in U_k$ can be chosen appropriately for $i, j, k \in \{0, 1\}$. There are thus $8$ of these, each contributing $m$ to the size of $|S||T||U|$.

\vspace{1mm}

A simple corollary is noted for some special cases.

\begin{corollary}
\textup{(1)} If $G$ is a group with an abelian normal subgroup $H$ of prime index $p$ and $(2p - 1) \nmid |G|$, that is, $(2p - 1)$ does not divide $|G|$, then $\rho_0(G) \leq \frac{1}{2}p$. \textup{(2)} If $G$ is a $p$-group with an abelian subgroup of index $p$ then $\rho_0(G) = 1$.
\end{corollary}
\begin{proof}\textup{(1)} By Theorem \ref{MABQuoPrime}, $|S||T||U| \leq \frac{p^2}{2p-1}|G| = \frac{p^3}{2p-1}|H|$ holds for any subgroup TPP triple $(S, T, U)$ of $G$, where $G$ and $H$ are given as above. If $(2p - 1) \nmid |G|$ then $(2p - 1) \nmid |H|$ and $|S||T||U| < \frac{p^2}{2p - 1}|G|$, which implies that $|S||T||U| \leq \frac{p^2}{2p}|G| = \frac{1}{2}p|G|$. This shows that if $2p - 1$ does not divide $|G|$ then $\rho_0(G) = 1$ when $p = 2$, and $\rho_0(G) > 1$ is possible only if $p \geq 3$.

\vspace{1mm}

\textup{(2)} Let $G$ be a (finite) $p$-group with an abelian subgroup $H$ of index $p$. Then $H$ is maximal in $G$, and thus normal \citep[Proposition 3, p.73]{AB}. And the size of any subgroup TPP triple of $G$ is a $p$-power. By Theorem \ref{MABQuoPrime} if $(S, T, U)$ is a subgroup TPP triple of $G$, then $|S||T||U|$ must be a $p$-power such that $|S||T||U| < \frac{p^2}{2p - 1}|G| < p|G|$, that is, $|S||T||U| \leq |G|$.
\end{proof}

The contrapositive of this corollary explains the non-trivial $\rho_0$ values computed by Hedtke and Murthy for all the groups listed in \citep[Tables 1-4]{HM} that achieve $\rho_0 > 1$ or $\rho_0 > \frac{1}{2}p$, and this has been verified by GAP computations in these groups. It also explains the trivial $\rho_0$ values listed in those tables for the $p = 2$ case where either $2\cdot2 - 1 = 3$ does not divide the order of an abelian subgroup of index $2$, as in the case of the dihedral groups $D_{2n}$ where $3$ does not divide $n$, or the group is a $2$-group containing an abelian subgroup of index $2$.

\section{Concluding Remarks}

In relation to the more general notion of TPP ratio $\rho$ the theorem means that in groups with abelian normal subgroups of prime index $p$ no three subgroups can realise the TPP with a triple size exceeding $\frac{p^2}{2p-1}|G|$. To achieve the latter one must look for triples of subsets at least one of which is not a subgroup.

\vspace{1mm}

In the case of $p = 2$ and the dihedral groups $D_{2n}$ the theorem proves that $\rho_0(D_{2n}) \leq \frac{4}{3}$ \citep[Conjecture 7.5]{HM}. The data for TPP triples in groups of order $\leq 32$ obtained via exhaustive computational search \citep[Table 1]{HM} shows that none of those groups realise $\rho > \frac{4}{3}$ if they contain a cyclic subgroup of index $2$, and this led to the original conjecture that $\rho \leq \frac{4}{3}$ for groups with cyclic subgroups of index $2$ \citep[Conjecture 7.6]{HM}.

\vspace{1mm}

The smallest known (nonabelian) group containing a cyclic normal subgroup of index $p > 2$ and a known $\rho > 1$ is one of type $C_7 \rtimes C_3$ (GAP ID [21, 1]), a group of order $21 = 3 \cdot 7$ containing a cyclic normal subgroup $C_7$ of index $p = 3$, and that realises a largest TPP triple of type $(3, 3, 3)$ and size $3^3 = 27$, and satisfies the bound $\frac{27}{21} < \frac{3^2}{2\cdot3 - 1} = \frac{9}{5} = 1.8$ \cite[Table 1]{HM}. The same data also shows that there are actually no groups of order $\leq 32$ containing cyclic normal subgroups of index $3$ and realising $\rho > \frac{9}{5}$. This seems to be true even for larger groups of order up to $100$, as indicated in the results of Xiang et. al. in 2018, who use a computational approach based on evolutionary search algorithms to look at TPP capacity $\beta$ (and other parameters related to group-theoretic matrix multiplication) in a selection of groups of order up to $100$ \citep[Tables 1, 3-4]{XZC}: some simple GAP computations show that of all groups listed in those tables that realise $\rho > \frac{9}{5}$ none contain cyclic normal subgroups of index $3$. Noting that subgroups of smallest prime index are always normal, this suggests the following conjecture (which is a generalisation of \citep[Conjecture 7.6]{HM}).

\begin{conjecture}
If $G$ is a group with a cyclic normal subgroup of prime index $p$, and $(S, T, U)$ is any TPP triple of $G$, then $|S||T||U| \leq \frac{p^2}{2p-1}|G|$, that is, $\rho(G) \leq \frac{p^2}{2p-1}$. For this bound to be as tight as possible $p$ can be taken to be the smallest such prime, in which case the cyclic subgroup of index $p$ will necessarily be normal.
\end{conjecture}

The conjecture is false for the more general case of abelian normal subgroups of prime index: for $p = 2$, where the upper bound to be respected is $\rho \leq \frac{4}{3}$, the smallest known counterexamples are two $2$-groups of order $32$, one of type $(C_4 \times C_4) \rtimes C_2$ (GAP ID [32, 11]) and another of type $(C_2 \times C_2 \times C_2 \times C_2) \rtimes C_2$ (GAP ID [32, 27]), both containing a non-cyclic abelian subgroup of index $2$, but no cyclic subgroups of index $2$, and realising $\rho = 1.5 > \frac{4}{3}$ (via TPP triples of type $(6, 4, 2)$).

\end{document}